\documentclass[12pt,reqno]{amsproc}

\usepackage[utf8]{inputenc}
\usepackage{mathpazo}
\usepackage{amsfonts}
\usepackage{amssymb}
\usepackage{amsmath}
\usepackage{amsthm}
\usepackage{mathrsfs}
\usepackage{bbm}
\usepackage{float}

\usepackage[margin=0.9in]{geometry}
\usepackage{graphicx}
\usepackage{tikz}
\usepackage{tikz-cd}
\usepackage[all,cmtip]{xy}

\usetikzlibrary{
    knots,
    decorations.markings
}

\tikzset{
    midarrow/.style={
        postaction=decorate,
        decoration={
            markings,
            mark=at position 0.5 with {\arrow{>}}
        }
    }
}

\usepackage{caption}
\usepackage{subcaption}

\usepackage{enumerate}
\usepackage{changes}
\usepackage{todonotes}
\usepackage{orcidlink}

\usepackage{hyperref}
\hypersetup{
    colorlinks=true,
    linkcolor=blue,
    citecolor=red,
    urlcolor=red,
    filecolor=blue
}

\usepackage[capitalise]{cleveref}

\newcommand{\x}{\times}
\newcommand{\g}{\Gamma}
\newcommand{\s}{\sigma}
\newcommand{\al}{\alpha}
\newcommand{\be}{\beta}
\newcommand{\ka}{\kappa}

\theoremstyle{plain}

\newtheorem{theorem}{Theorem}[section]
\newtheorem*{theorem*}{Theorem}
\newtheorem{lemma}[theorem]{Lemma}

\theoremstyle{definition}

\newtheorem*{example*}{Example}
\newtheorem{remark}[theorem]{Remark}

\title{Borel--Moore Homology of 2-Particle Unordered Configuration Spaces of Graphs via Discrete Morse Theory}
\author{Munawar Khan}

\begin{document}
\begin{abstract}

We define a discrete Morse function on the \emph{open cell complex}
\(C_2(\g)\), arising from the 2-particle unordered configuration space of a graph \(\g\). We compute the homology of the associated Morse complex. Using an isomorphism established by Knudson and Scoville between the Borel--Moore homology of \(C_2(\g)\) and the homology of the Morse complex, we conclude that the Borel-Moore homology groups of \(C_2(\g)\) are completely determined by the first Betti number of \(\g\).
\end{abstract}

\maketitle
\markboth{Munawar Khan}{}

\section*{Introduction}
Discrete Morse theory, as introduced by Forman \cite{Forman1998,Forman2002}, gives an analogue of classical Morse theory in discrete settings for cell complexes. This theory was widely developed and has since been applied in several areas, including \cite{KnudsonWang2022,JostZhang2024,Benedetti2012}. It has also been studied in the setting of \emph{open cell complexes} \cite{KnudsonScoville2026}.

In \cite{Sawicki2012}, see also \cite{Farley_2005}, Sawicki systematically defined a discrete Morse function on the \(2\)-particle combinatorial configuration space \(D^2(\g)\) for connected and simple graphs. We extend this idea to define a discrete Morse function on the open cell complex \(C_2(\g)\). In this paper, an \emph{open cell complex} is the complement of a nonempty subcomplex of a regular CW complex.
Since \(C_2(\g)\) and \(D^2(\g)\) are homotopy equivalent \cite{Sawicki2012,Abrams2000ConfigurationSpacesGraphBraidGroups,Prue_2014}, their singular homology groups are isomorphic.Thus, computing \(H_n(C_2(\g))\) reduces to defining a discrete Morse function on \(D^2(\g)\) and computing the homology of the resulting Morse complex.

However, Borel--Moore homology is not invariant under arbitrary homotopies and, in general, does not agree with singular homology for non-compact spaces\cite{zbMATH03189968}. Instead of \(D^2(\g)\), we define a discrete Morse function directly on \(C_2(\g)\) and use the result of \cite{KnudsonScoville2026}, which shows that the homology of the Morse complex associated to a discrete Morse function on an open cell complex is isomorphic to its Borel--Moore homology. We therefore apply discrete Morse theory for open cell complexes to \(C_2(\g)\), allowing us to compute its Borel--Moore homology combinatorially.
Most of this work follows the approach of \cite{Sawicki2012} and is a natural extension of this construction.

\textbf{Notation:} Throughout the article, \(\g\) will denote a connected, simple graph. Let \(\operatorname{Sym}^2(\g)\) be the symmetric space of \(\g\), defined as follows:
\[
\operatorname{Sym}^2(\g)=\left(\g\times\g\right)/\sim,
\]
where
\[
(x_1,y_1)\sim(x_2,y_2)\iff (x_1,y_1)=(y_2,x_2) \text{ or } (x_1,y_1)=(x_2,y_2).
\]
Let
\[
\Delta=\{(a,a)\mid a\in\g\}\subseteq\g\times\g
\]
denote the diagonal of \(\g\times\g\). We note that \(\operatorname{Sym}(\g)\) has a regular CW complex structure and \(\Delta\) is a nonempty subcomplex.
The \(2\)-particle unordered configuration space of \(\g\) is defined as
\[
C_2(\g)=\operatorname{Sym}^2(\g)-\Delta ,
\]

Whenever \(e=(i,j)\) is a \(1\)-cell in \(\g\), we introduce a \(1\)-cell along the diagonal with faces \(i\x i\) and \(j\x j\) and call it \(D(e)\) or \(D(i,j)\), as appropriate. We consider one triangular half of the \(2\)-cell \(e\x e\) and we call it \(C_2(e)\) or \(C_2(e(v))\), as appropriate. Here, we take the lower half, although equivalently one could take the upper half. The \(2\)-cell \(C_2(e)\) is shown in the following diagram:
\par\noindent
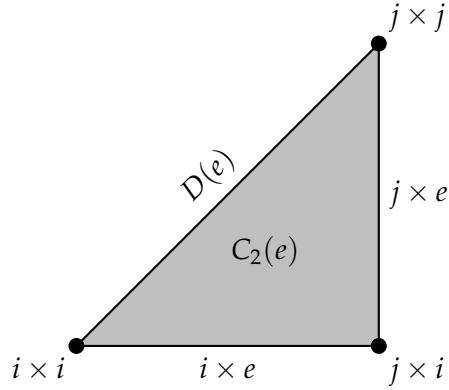
\begin{figure}[ht]
    \centering
    \begin{tikzpicture}
    
        \fill[gray!50] (0,0) -- (4,0) -- (4,4) -- cycle;
        
        \draw[thick] (0,0) -- (4,0) node[midway, below] {\small \(i\x e\)};
        \draw[thick] (4,0) -- (4,4) node[midway, right] {\small \(j\x e\)};
        \draw[thick] (4,4) -- (0,0) node[midway, sloped, above] {\small \(D(e)\)};
       
        \fill[black] (0,0) circle (3pt);
        \fill[black] (4,0) circle (3pt);
        \fill[black] (4,4) circle (3pt);
       
        \node[below left] at (0,0) {\small \(i\x i\)};
        \node[below right] at (4,0) {\small \(j\x i\)};
        \node[above right] at (4,4) {\small \(j\x j\)};
       
        \node at (2.5,1.25) {\small \(C_2(e)\)};
        
    \end{tikzpicture}
    \caption{
    The \(2\)-simplex \(C_2(e)\).}
    \label{fig:C2e}
\end{figure}
\par

We define a discrete Morse function \(\tilde f_2\) on \( \operatorname{Sym}^2(\g) \) such that all cells in \( \Delta \) are critical. This gives a discrete Morse function on the open cell complex \( C_{2}(\g) \)\cite{KnudsonScoville2026}. Then we compute the Morse complex \(\mathcal{M}(f)\) for this function and compute Borel--Moore homology using Theorem \ref{Thm:knudson}.

\section{Preliminaries}
In this section, we briefly introduce the elements of discrete Morse theory that will be used throughout the article.
We skip proofs and details that are not directly needed and refer interested readers to the following sources \cite{Forman1998,Scoville2019,KnudsonScoville2026}.
\subsection{Discrete Morse Theory for Regular CW Complexes}
Let \(M\) be a regular CW complex. For cells \(\sigma,\tau\in M\), we write \(\tau>\sigma\) if \(\sigma\) is a face of \(\tau\), and \(\tau<\sigma\) if \(\tau\) is a face of \(\sigma\).
 A function \(f:M\to\mathbb{R}\) is called a \emph{discrete Morse function} if, for every \(p\)-cell \(\sigma\), the following inequalities hold:
\begin{equation}
\label{eq:DMF_definition_inequality1}
\#\{\tau^{(p+1)}>\sigma:f(\sigma)\geq f(\tau)\}\leq 1,
\end{equation}
and
\begin{equation}
\label{eq:DMT_definition_inequality2}
\#\{v^{(p-1)}<\sigma:f(v)\geq f(\sigma)\}\leq 1.
\end{equation}

A cell \(\sigma^{(p)}\) is called a \emph{critical \(p\)-cell} if both inequalities \eqref{eq:DMF_definition_inequality1} and \eqref{eq:DMT_definition_inequality2} are strict.

Here we briefly recall the following standard facts from discrete Morse theory \cite{Forman1998}. For every non-critical cell \(\sigma\), exactly one of the following holds:
\begin{equation}
\label{exclusive condition 1}
\exists\,\tau^{(p+1)}>\sigma \quad \text{such that}\quad f(\tau)\leq f(\sigma),
\end{equation}
or
\begin{equation}
\label{exclusive condition 2}
\exists\,v^{(p-1)}<\sigma \quad \text{such that}\quad f(v)\geq f(\sigma).
\end{equation}

Given $a\in\mathbb{R}$, the associated level subcomplex is
\[
M(a)=\bigcup_{\s\leq\tau,\ f(\tau)\leq a}\s \,\,.
\]
If an interval \( [a,b] \) does not contain any non-critical values, then
\[
M(b)\searrow M(a).
\]

If the interval $[a,b]$ contains exactly one critical cell $\s^{(p)}$,
then $M(b)$ is homotopy equivalent to the complex obtained from $M(a)$
by attaching a $p$-cell:
\[
M(b)\simeq M(a)\cup e^{(p)}.
\]

Consequently, if $m_p(f)$
denotes the number of critical $p$-cells, then $M$ is homotopy equivalent
to a CW complex with $m_p(f)$ cells of dimension $p$.

A discrete vector field $V$ on $M$ is a collection of pairs
\[
\{\s^{(p)}<\tau^{(p+1)}\},
\]
such that every cell appears in at most one pair. Pictorially, we draw an arrow with \(\sigma\) as its tail and \(\tau\) as its head. A $V$-path is a sequence
\[
\s_0^{(p)},\tau_0^{(p+1)},\s_1^{(p)},\tau_1^{(p+1)},\ldots,
\s_n^{(p)}
\]
where $\{\s_i,\tau_i\}\in V$ and $\tau_i>\s_{i+1}$.
A discrete Morse function determines a discrete vector field called an \emph{induced gradient vector field}, defined by the pairs
\[
\{\s^{(p)}<\tau^{(p+1)}\},\qquad f(\s)\geq f(\tau).
\]
Conversely, a discrete vector field is induced by a discrete Morse function if and only if it contains no closed $V$-paths

Before going further, we define the notion of the Morse complex. Let \(f\) be a discrete Morse function on an oriented regular CW complex \(M\). The corresponding Morse complex is defined as follows:
\(\mathcal{M}_p(f)\) is the free group generated by the critical cells of dimension \(p\) in \(M\). The Morse boundary map \(\partial_p:\mathcal{M}_p(f)\longrightarrow\mathcal{M}_{p-1}(f)\) is defined as follows:
Given a critical cell \(\tau^{(p)}\)
\[
\partial \,\tau = \sum_{\s^{(p-1)}}\left(\sum_{\tilde{\s}^{(p-1)}<\tau}{i_{\tilde{\s},\tau}\cdot \mu(\gamma(\tilde{\s},\s))}\right)\s .
\]
Where, the outer sum runs over all critical cells of dimension \(p-1\), the inner sum runs over all the faces \(\tilde{\s}\) of \(\tau\), the innermost sum runs over all gradient upper paths \(\gamma(\tilde{\s},\s)\) from \(\tilde{\s}\) to \(\s\). The term \(\mu(\gamma(\tilde{\s},\s))\) is a number called the multiplicity of the path \(\gamma(\tilde{\s},\s)\), determined based on the orientation and incidence numbers of the cells in the path. This number for a regular CW complex is always \(\pm 1\).

\subsection{Discrete Morse Theory for Open Complexes}

Given an open complex $K=X - T$, where $X$ is a simplicial
complex and $T$ is a nonempty subcomplex, discrete Morse theory
can be defined via a discrete gradient or a discrete Morse function.
However, the standard conditions for a discrete Morse function may not be
mutually exclusive on open complexes, requiring additional constraints to
resolve ambiguity in the induced discrete gradient
\cite{KnudsonScoville2026}.

Unlike the simplicial setting, discrete Morse theory on open complexes exhibits several anomalous behaviors \cite{KnudsonScoville2026}.
Local minima do not necessarily occur on $0$-cells, and adding a regular
cell can trigger topological changes. Furthermore, a single critical cell
may simultaneously create and destroy homology classes in multiple
dimensions, and a nonempty complex may lack critical cells entirely,
potentially leading to the conclusion that the space is homotopy
equivalent to an empty set
\cite{KnudsonScoville2026}. Because of these anomalies, the
homology of the Morse complex for an open complex is not isomorphic to its
singular homology; instead, it is isomorphic to its Borel--Moore homology
$H_*^{\mathrm{BM}}(K)$ \cite{KnudsonScoville2026}.
\begin{theorem}\cite{KnudsonScoville2026}
\label{Thm:knudson}
For all \(n \geq 0\), there is an isomorphism
\[
H_n(\mathcal{M}(f)) \cong H^{BM}_n(K).
\]
Where \(\mathcal{M}(f)\) is the Morse complex corresponding to a discrete Morse function \(f\) on \(K\).
\end{theorem}

Although the definition of an open cell complex in \cite{KnudsonScoville2026} is given in the setting of simplicial complexes, the authors note in Theorem 3.1 and Example 3.4 that Theorem \ref{Thm:knudson} also holds for open cell complexes arising from regular CW complexes.

Borel--Moore homology \cite{zbMATH03189968} is a homology theory for locally
compact spaces that extends singular homology by allowing locally finite
chains. These chains are formal sums of singular simplices such that every
compact subset of the space intersects only finitely many simplices appearing
with nonzero coefficient \cite{zbMATH03189968}. For compact spaces, this theory agrees with
singular homology. Unlike singular homology, Borel--Moore homology is a
covariant functor for \emph{proper maps} and is invariant under
\emph{proper homotopies} \cite{zbMATH03189968}.

A continuous map $f:X\to Y$ is \emph{proper} if the preimage of every
compact subset $K\subseteq Y$ is compact:
\[
K\subseteq Y \text{ compact } \implies f^{-1}(K)\subseteq X
\text{ is compact}.
\]
If two proper maps $f_0,f_1:X\to Y$ are connected by a proper homotopy, then they induce the same map on Borel--Moore homology:
\[
(f_0)_*=(f_1)_*:H_k^{\mathrm{BM}}(X)\longrightarrow
H_k^{\mathrm{BM}}(Y).
\]

We recall the definition of a discrete Morse function on an open complex.
Let \(X\) be a regular CW complex, \(T\) a nonempty subcomplex, and
\(K=X-T\). A function \(f:K\rightarrow\mathbb{R}\) is called an open
discrete Morse function(or just a discrete Morse function when from the context it is clear that \(K\) is an open cell complex) if, for every simplex \(\s^{(p)}\in K\),
\[
\#\{\tau^{(p+1)}>\s:\tau\in K,\ f(\tau)\leq f(\s)\}\leq 1,
\]
and
\[
\#\{v^{(p-1)}<\s:v\in K,\ f(v)\geq f(\s)\}\leq 1,
\]
with at least one of the above inequalities being strict
\cite{KnudsonScoville2026}.

\section{A Discrete Morse Function on \texorpdfstring{$C_2(\g)$}{C2(g)}}
We start by considering a connected simple graph \(\g\) and a spanning tree \(T\) of \(\g\). That is, \(T\) is a connected subgraph of \(\g\) such that \(V(\g)=V(T)\) and there is a unique path in
\(T\) between each pair of vertices of \(\g\).
Now we label the graph \(\g\) as follows;
Start with a vertex of valency \(1\), i.e., a vertex that is connected to only one edge in \(\g\); label it \(1\). Now traverse the tree anticlockwise and label the vertices \(i\) in increasing order. An example of such a labelling can be seen in Figure \ref{fig:house_graph}, where the solid lines represent the edges in \(T\), and the dotted lines represent the edges in \(\g - T\).
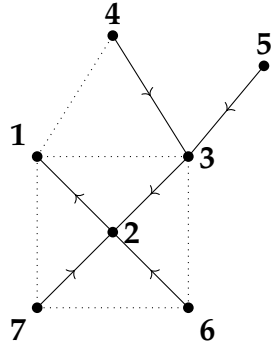
\begin{figure}[ht]
    \centering
    \begin{tikzpicture}[scale=2]

        \fill (0,0) circle (1pt); 
        \node[below left] at (0,0) {\textbf{7}};
        \fill (1,0) circle (1pt); 
        \node[below right] at (1,0) {\textbf{6}};
        \fill (1,1) circle (1pt); 
        \node[above right] at (1,0.85) {\textbf{3}};
        \fill (0,1) circle (1pt); 
        \node[above left] at (0,1) {\textbf{1}};
        \fill (0.5,1.8) circle (1pt); 
        \node[above] at (0.5,1.8) {\textbf{4}};
        \fill (1.5,1.6) circle (1pt); 
        \node[above] at (1.5,1.6) {\textbf{5}};
        \fill (0.5,0.5) circle (1pt); 
        \node[right] at (0.5,0.5) {\textbf{2}};

        \draw[dotted] (0,0) -- (1,0);
        \draw[dotted] (1,0) -- (1,1);
        \draw[dotted] (1,1) -- (0,1);
        \draw[dotted] (0,1) -- (0,0);

        \draw[midarrow] (0,0) -- (0.5,0.5);
        \draw[midarrow] (1,0) -- (0.5,0.5);
        \draw[midarrow] (1,1) -- (0.5,0.5);
        \draw[midarrow]  (0.5,0.5) -- (0,1);

        \draw[dotted] (0,1) -- (0.5,1.8);
        \draw[midarrow]  (1.5,1.6) -- (1,1);
        \draw[midarrow] (0.5,1.8) -- (1,1);

    \end{tikzpicture}
    \caption{Labeled house graph}
    \label{fig:house_graph}
\end{figure}

Now, we define a discrete Morse function on \(\g\) as follows:
\[
f_1:\g \longrightarrow \mathbb{R},
\]
\[f_1(i):= 2i-2,\]
\[
f_1((i,j)):=
\begin{cases}
\max\{f_1(i),f_1(j)\} & \text{if } (i,j) \in T,\\
 \max\{f_1(i),f_1(j)\}+2& \text{if } (i,j)\in \g - T.
\end{cases}
\]

\begin{remark}
    \(f_1\) is a perfect discrete Morse function on \(\g\) \cite{Sawicki2012}.
\end{remark}
We further define a function \(f_2:X \longrightarrow\mathbb{R}\) as follows:
for cells in \(C_2(\g)\)
\[
\begin{aligned}
0\text{-cells}: \quad & {f}_{2}(i \x j) = f_{1}(i) + f_{1}(j), \\
1\text{-cells}: \quad & {f}_{2}(i \x (j,k)) = f_{1}(i) + f_{1}((j,k)), \\
2\text{-cells}: \quad & {f}_{2}((i,j) \x (k,l)) = f_{1}((i,j)) + f_{1}((k,l)).
\end{aligned}
\]
for cells in \(\Delta\)
\[
\begin{aligned}
0\text{-cells}: \quad & {f}_{2}(i \x i) = -4, \\
1\text{-cells}: \quad & {f}_{2}(D(i,j)) = -2.
\end{aligned}
\]

This function, in general, is not a discrete Morse function. We check the definition of a discrete Morse function for each cell and modify the values of \(f_2\) on the cells that do not satisfy the definition. We denote the resulting function by \(\tilde{f}_2:X\rightarrow\mathbb{R}\), which is a discrete Morse function. For the cells \(\s\) which already satisfy the conditions for a discrete Morse function, we define \(\tilde{f}_2(\s)=f_2(\s)\).

A discrete Morse function on \(\operatorname{Sym}^2(\g)\) with all cells in \(\Delta\) critical, when restricted to \(C_2(\g) = X-\Delta\), satisfies the definition of a discrete Morse function on the open cell complex \(C_2(\g)\)
\cite{KnudsonScoville2026}.

We will use this fact to define a discrete Morse function on \(C_2(\g)\).
We further define the Morse complex corresponding to such open discrete Morse functions in the standard way, considering only cells on \(K\).

If each \(2\)-cell in \(C_2(\g)\) has a non-negative value, then each \(1\)-cell \(D(i,j)\) in \(\Delta\) satisfies the definition of a discrete Morse function and is critical. Similarly, if each \(1\)-cell in \(C_2(\g)\) has a non-negative value, then each
\(0\)-cell \((i,i)\in\Delta\) satisfies the definition and is critical.

Now, we only need to check the cells in \( C_2(\g) \), because cells in \( \Delta \) automatically satisfy the definition, given that the cells in \( C_2(\g) \) have non-negative values.

Now, we check the definition for cells in \(C_2(\g)\).
We note that for each edge \(e \in T\) there exists a unique \(0\)-cell \(v\) of \(e\) with \(f_1(e)=f_1(v)\), we represent such edges by \(e(v)\), we further note that there is no such edge for the vertex \(1\), therefore there is a one to one correspondence 
\[
\begin{aligned}
\operatorname{V}(\g)-\{1\}
&\longleftrightarrow
E(T),\\
v &\longleftrightarrow e(v).
\end{aligned}
\]
 Also, each vertex \(v\) can have edges in \(\g - T\), connected with \(v\), we represent set of such edges by \(D_v\) and the edges in \(T\) different from \(e(v)\) connected to \(v\) by \(T_v\).
 \begin{figure}[ht]
\centering
\begin{tikzpicture}[>=stealth,scale=1.2]

\node[circle,fill,inner sep=1.8pt,label=below:\(v\)] (v) at (0,0) {};

\draw[midarrow] (-1.7,1.4) -- (v);
\draw[midarrow] (-1.85,1) -- (v);
\draw[midarrow] (-1.45,1.65) -- (v);

\draw[midarrow] (v) -- (1.6,0.6);

\draw (-1.7,-1.4) -- (v);
\draw (-1.85,-1) -- (v);
\draw (-1.45,-1.65) -- (v);

\node at (-2.0,1.4) {\(T_v\)};
\node at (-2.0,-1.4) {\(D_v\)};
\node at (1.6,0.8) {\(e(v)\)};

\end{tikzpicture}

\caption{Edges around vertex \(v\)}
\label{fig:vertex-v}

\end{figure}
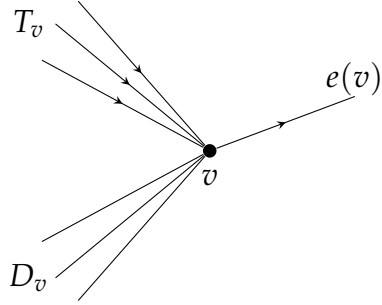
 Note that \(e(v)\) is the unique edge with an arrow directed away from \(v\) in the discrete gradient vector field induced by \(f_1\), while \(T_v\) consists of the edges with arrows directed towards \(v\), and \(D_v\) consists of the edges incident to \(v\) that are not paired in the discrete gradient vector field, as shown in Figure~\ref{fig:vertex-v}.
 
\textbf{Step 1.}
We partition the set of \(2\)-cells in \(C_2(\g)\) as follows:
\begin{enumerate}
    \item[1.] \(\al = e_1\x e_2 \) with \(e_1, e_2 \in \g - T\) and \(e_1 \neq e_2\).
    \item[2.] \(\al = C_2(e) \) with \(e \in \g - T\).
    \item[3.] \(\al = e\x e(u) \) with \(e\in \g - T, e(u) \in  T\).
    \item[4.] \(\al = e(u)\x e(v) \) with \(e(u), e(v) \in  T\) and \(e(u) \neq e(v)\).
    \item[5.] \(\al = C_2(e(v)) \) with \( e(v) \in  T\).
\end{enumerate}

For each cell in Cases \(1-5\), we check the following condition:
\begin{equation}
\label{cell_2_condition}
\#\{ \be^{(1)}\in C_2(\g)\,|\, \be<\al \text{ and }f_2(\be)\geq f_2(\al) \}\leq 1.
\end{equation}

\textbf{Step 2.}
We partition the set of \(1\)-cells in \(C_2(\g)\) as follows:
\begin{enumerate}
    \item[6.] \(\be = v\x e(u)\) with \(e(u)\in T\) and \(e(u)\cap e(v) \neq \emptyset \).
    \item[7.] \(\be = v\x e\) with \(e \in \g - T\) and \(e(v)\cap e \neq \emptyset \).
    \item[8.] \(\be = v\x e(u)\) with \(e(u)\in T\) and \(e(u)\cap e(v) = \emptyset \).
    \item[9.] \(\be = v\x e\) with \(e \in \g - T\) and \(e(v)\cap e = \emptyset \).
\end{enumerate}

For each cell in Cases \(6-9\), we check the following conditions:
\begin{equation}
\label{cell_1_condition_1}
\#\{\al^{(2)}\in C_2(\g)\,|\,\be<\al \text{ and }f_2(\be)\geq f_2(\al) \}\leq 1,
\end{equation}
and 
\begin{equation}
\label{cell_1_condition_2}
\#\{\ka^{(0)}\in C_2(\g)\,|\, \ka<\be \text{ and }f_2(\ka)\geq f_2(\be) \}\leq 1.
\end{equation}

\textbf{Step 3.}
We partition the set of \(0\)-cells in \(C_2(\g)\) as follows:
\begin{enumerate}
    \item[10.] \(\ka = u\x v \) with \(u\neq v\) and \(u\neq 1, v \neq 1\).
    \item[11.] \(\ka = 1\x v \) with \(v\neq 1\).
\end{enumerate}
For each cell in Cases \(10-11\), we check the following condition:
\begin{equation}
\label{cell_0_condition}
\#\{\be^{(1)} \in C_2(\g)\,|\,\ka<\be \text{ and }f_2(\ka)\geq f_2(\be) \}\leq 1.
\end{equation}

With a modification of \(f_2\) on  \(2\)-cells in Case 4 in \textbf{Step 1}, i.e. \(\al = e(u)\x e(v)\) with \(e(u)\neq e(v)\), and tracking these changes in \(\tilde{f}_2\), we get a discrete Morse function.

We also set the convention that whenever \(e\) is a \(1\)-cell in \(T\), we write \(e=(v,\tau{(v)})\), where \(v\) is the unique face of \(e\) with \(e=e(v)\) and \(\tau{(v)}\) is the terminal vertex of \(e\), i.e. \(\tau(v)\) is the unique face of \(e\) with \(f_1(\tau(v))<f_1(e)\). Similarly, whenever \(e\) is a \(1\)-cell in \(\g - T\), we write \(e = (i,j)\), where \(i, j\) are faces of \(e\) and we always assume \(i<j\).

\begin{lemma}
    The \(2\)-cells of the form \(\al = e_1\x e_2 \), where \(e_1, e_2 \in \g - T\) and \(e_1 \neq e_2\) satisfy condition \eqref{cell_2_condition} and are critical.
\end{lemma}

\begin{proof}
Let \(e_1 = (i,j)\) and \(e_2 = (k,l)\). The \(2\)-cell \(\al\) and the values of \(f_2\) on its faces are given in the following diagram:
\par\noindent
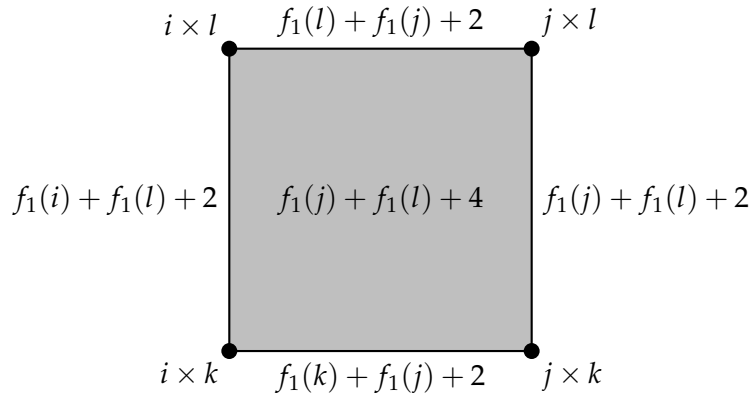
\centering
    \begin{tikzpicture}
    
        \fill[gray!50] (0,0) rectangle (4,4);
        
        \draw[thick] (0,0) -- (4,0) node[midway, below] {\small \(f_1(k) + f_1(j) + 2\)};
        \draw[thick] (4,0) -- (4,4) node[midway, right] {\small \(f_1(j) + f_1(l)+2\)};
        \draw[thick] (4,4) -- (0,4) node[midway, above] {\small \(f_1(l)+f_1(j)+2\)};
        \draw[thick] (0,4) -- (0,0) node[midway, left] {\small \(f_1(i)+f_1(l)+2\)};
    
        \fill[black] (0,0) circle (3pt);
        \fill[black] (4,0) circle (3pt);
        \fill[black] (4,4) circle (3pt);
        \fill[black] (0,4) circle (3pt);
    
        \node[below left] at (0,0) {\small \(i\x k\)};
        \node[below right] at (4,0) {\small \(j\x k\)};
        \node[above right] at (4,4) {\small \(j \x l\)};
        \node[above left] at (0,4) {\small \(i \x l\)};
    
        \node at (2,2) {\small \(f_1(j)+f_1(l)+4\)};
    \end{tikzpicture}
\par

Note that \(f_1(i)<f_1(j)\) and \(f_1(k)<f_1(l)\). Therefore, for each face \(\be^{(1)}<\al\), we have \(f_2(\be)<f_2(\al)\). Hence, the condition \eqref{cell_2_condition} is satisfied and \(\al\) is a critical cell.
\end{proof}

\begin{lemma}
    The \(2\)-cells of the form \(\al = C_2(e)\), where \(e\in \g - T\), satisfy condition \eqref{cell_2_condition} and are critical.
\end{lemma}
\begin{proof}
This proof is similar to the previous proof.
Let \(e=(i,j)\), the \(2\)-cell \(\al\) and the values on its faces are given in the following diagram:
\par\noindent
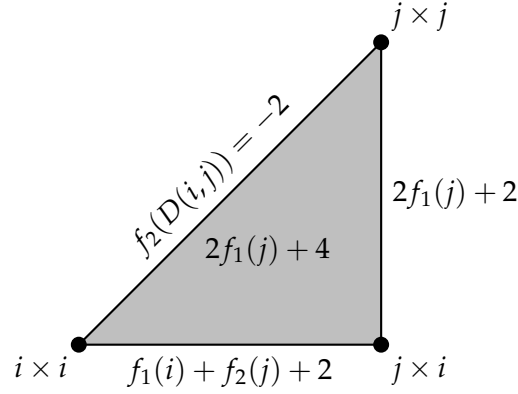
\centering
    \begin{tikzpicture}
    
        \fill[gray!50] (0,0) -- (4,0) -- (4,4) -- cycle;
        
        \draw[thick] (0,0) -- (4,0) node[midway, below] {\small \(f_1(i) + f_2(j) + 2\)};
        \draw[thick] (4,0) -- (4,4) node[midway, right] {\small \(2f_1(j)+2\)};
        \draw[thick] (4,4) -- (0,0) node[midway, sloped, above] {\small \(f_2(D(i,j))=-2\)};
       
        \fill[black] (0,0) circle (3pt);
        \fill[black] (4,0) circle (3pt);
        \fill[black] (4,4) circle (3pt);
       
        \node[below left] at (0,0) {\small \(i\x i\)};
        \node[below right] at (4,0) {\small \(j\x i\)};
        \node[above right] at (4,4) {\small \(j \x j\)};
       
        \node at (2.5,1.25) {\small \(2f_1(j)+4\)};
    \end{tikzpicture}
\par

Note that \(f_1(i)<f_1(j)\). Therefore, for each face \(\be^{(1)}<\al\), we have \(f_2(\be)<f_2(\al)\). Hence, the condition \eqref{cell_2_condition} is satisfied and \(\al\) is a critical cell.
\end{proof}

\begin{lemma}
\label{lemma4}
    The \(2\)-cells of the form \(\al = e\x e(v) \), where \(e\in \g - T,\, e(v) \in  T\) satisfy condition \eqref{cell_2_condition} and are non-critical.
\end{lemma}

\begin{proof}
Let \(e = (i,j)\) and \(e(v) = (v,\tau(v))\). The \(2\)-cell \(\al\) and values on its faces are given in the following diagram:
\par\noindent
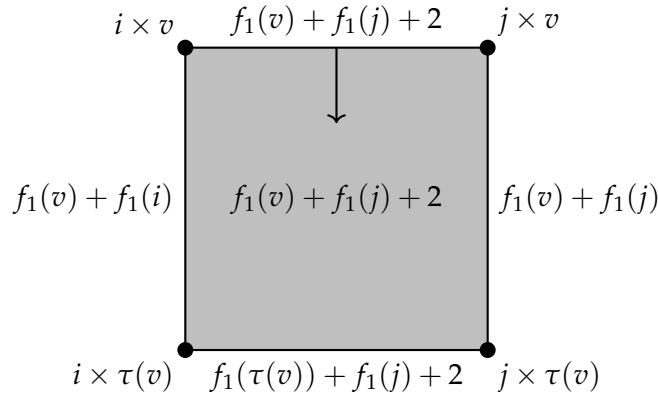
\centering
    \begin{tikzpicture}
    
        \fill[gray!50] (0,0) rectangle (4,4);
        
        \draw[thick] (0,0) -- (4,0) node[midway, below] {\small \(f_1(\tau(v)) + f_1(j) + 2\)};
        \draw[thick] (4,0) -- (4,4) node[midway, right] {\small \(f_1(v) + f_1(j)\)};
        \draw[thick] (4,4) -- (0,4) node[midway, above] {\small \(f_1(v)+f_1(j)+2\)};
        \draw[thick] (0,4) -- (0,0) node[midway, left] {\small \(f_1(v)+f_1(i)\)};
        \draw[thick,->] (2,4) -- (2,3);
    
        \fill[black] (0,0) circle (3pt);
        \fill[black] (4,0) circle (3pt);
        \fill[black] (4,4) circle (3pt);
        \fill[black] (0,4) circle (3pt);
    
        \node[below left] at (0,0) {\small \(i\x \tau(v)\)};
        \node[below right] at (4,0) {\small \(j\x \tau(v)\)};
        \node[above right] at (4,4) {\small \(j \x v\)};
        \node[above left] at (0,4) {\small \(i \x v\)};
    
        \node at (2,2) {\small \(f_1(v)+f_1(j)+2\)};  
    \end{tikzpicture}
\par
Note that \(f_2(v\x e) = f_2(\al)\), and for the other faces \(\be^{(1)}<\al\), we get \(f_2(\be)< f_2(\al)\), because \(f_1(i)<f_1(j)\) and \(f_1(\tau(v))<f_1(v)\). Therefore, the condition \eqref{cell_2_condition} is satisfied and \(\al\) is a non-critical cell.
\end{proof}

\begin{lemma}
\label{lemma5}
    For the \(2\)-cells of the form \(\al = e(u)\x e(v)\), where
    \(e(u), e(v) \in T\) and \(e(u) \neq e(v)\), there exist exactly two
    \(1\)-faces, namely \(u\x e(v)\) and
    \(v\x e(u)\), such that
    \(f_2(\al)=f_2(u\x e(v))=f_2(v\x e(u))\). Without loss of generality assume \(u< v\), we define
    \(\tilde{f}_2(\al)=f_2(\al)+1\) and
    \(\tilde{f}_2(u\x e(v))=f_2(u\x e(v))+1\). After this modification,
    \(\al\) satisfies the condition \eqref{cell_2_condition} and is
    non-critical.
\end{lemma}

\begin{proof}
Let \(e(u) = (u,\tau(u))\) and \(e(v) = (v,\tau(v))\). The \(2\)-cell \(\al\) and values on its faces is given in the following diagram:
\par\noindent
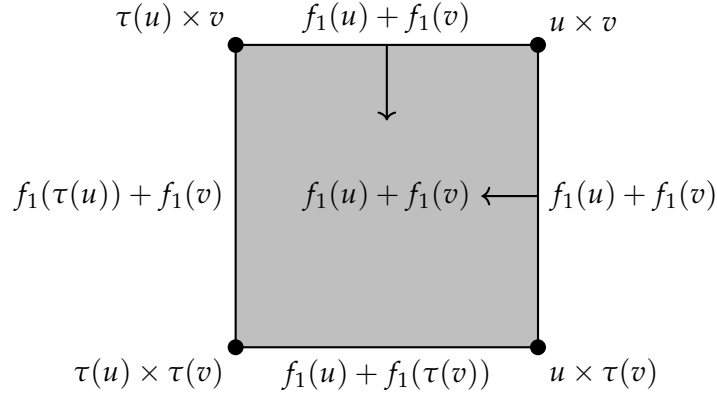
\centering
    \begin{tikzpicture}
        \fill[gray!50] (0,0) rectangle (4,4);
        
        \draw[thick] (0,0) -- (4,0) node[midway, below] {\small \(f_1(u) + f_1(\tau(v)) \)};
        \draw[thick] (4,0) -- (4,4) node[midway, right] {\small \(f_1(u) + f_1(v)\)};
        \draw[thick] (4,4) -- (0,4) node[midway, above] {\small \(f_1(u)+f_1(v)\)};
        \draw[thick] (0,4) -- (0,0) node[midway, left] {\small \(f_1(\tau(u))+f_1(v)\)};
        \draw[thick,->] (2,4) -- (2,3);
        \draw[thick,->] (4,2) -- (3.25,2);
    
        \fill[black] (0,0) circle (3pt);
        \fill[black] (4,0) circle (3pt);
        \fill[black] (4,4) circle (3pt);
        \fill[black] (0,4) circle (3pt);
    
        \node[below left] at (0,0) {\small \(\tau(u)\x \tau(v)\)};
        \node[below right] at (4,0) {\small \(u\x \tau(v)\)};
        \node[above right] at (4,4) {\small \(u \x v\)};
        \node[above left] at (0,4) {\small \(\tau(u) \x v\)};
    
        \node at (2,2) {\small \(f_1(u)+f_1(v)\)};
    \end{tikzpicture}
\par

Considering that \(f_1(\tau(u))<f_1(u)\) and \(f_1(\tau(v))<f_1(v)\), we note that there are two \(1\)-cells \(\be_1 = u\x e(v)\) and \(\be_2 = v \x e(u)\) with \(f_2(\be_i)\geq f_2(\al)\). In fact \(f_2(\al) = f_2(\be_1) = f_2(\be_2)\). We modify the value of \(\al\) along with \(\be_1\) as follows:
\(\tilde{f}_2(\be_1)=\tilde{f}_2(\al) = f_2(\al)+1\). Therefore, after this modification, the condition \eqref{cell_2_condition} is satisfied and \(\al\) is non-critical.
\end{proof}
Note that we could equivalently modify the value of \(v\x e(u)\). However, modifying the value of \(u\x e(v)\) as above is sufficient to obtain a discrete Morse function.

\begin{lemma}
\label{lemma6}
    The \(2\)-cells of the form \(\al = C_2(e(v))\), where \(e(v) \in T\) satisfy condition \eqref{cell_2_condition} and are non-critical.
\end{lemma}
\begin{proof}

This proof is similar to the previous proof. Let \(e(v) = (v,\tau(v))\), the \(2\)-cell \(\al\) and values of \(f_2\) on its faces is given in the following diagram:
\par\noindent
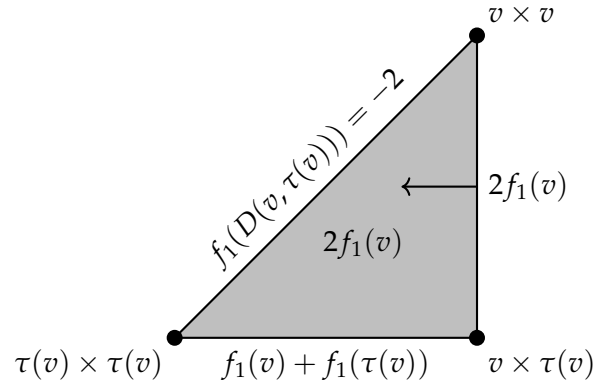
\centering
    \begin{tikzpicture}
    
        \fill[gray!50] (0,0) -- (4,0) -- (4,4) -- cycle;
        
        \draw[thick] (0,0) -- (4,0) node[midway, below] {\small \(f_1(v) + f_1(\tau(v))\)};
        \draw[thick] (4,0) -- (4,4) node[midway, right] {\small \( 2f_1(v)\)};
        \draw[thick] (4,4) -- (0,0) node[midway, sloped, above] {\small \(f_1(D(v,\tau(v)))=-2\)};
        \draw[thick,->] (4,2) -- (3,2);
       
        \fill[black] (0,0) circle (3pt);
        \fill[black] (4,0) circle (3pt);
        \fill[black] (4,4) circle (3pt);
       
        \node[below left] at (0,0) {\small \(\tau(v)\x \tau(v)\)};
        \node[below right] at (4,0) {\small \(v\x \tau(v)\)};
        \node[above right] at (4,4) {\small \(v \x v\)};
       
        \node at (2.5,1.25) {\small \(2f_1(v)\)};
    \end{tikzpicture}
\par
Note that \(f_1(\tau(v))<f_1(v)\), therefore \(f_2(\tau(v)\x e(v)) < f_2(\al)\) and \(f_2(v\x e(v)) = f_2(\al)\), therefore the condition \eqref{cell_2_condition} is satisfied and \(\al\) is a non-critical cell.
\end{proof}

For each \(1\)-cell, say \(\be\), in Cases 6-9, we will consider the \(2\)-cells, say \(\al\), such that \(\be\) could occur as a face of \(\al\). We will first check whether the values of \(f_2\) on \(\al\) or \(\be\) were modified in Lemma \ref{lemma5}, and then check the conditions \eqref{cell_1_condition_1} and \eqref{cell_1_condition_2} accordingly.
 
\begin{lemma}
    The \(1\)-cells of the form \(\be = v\x e(u)\), where \(e(u)\in T\) and \(e(v)\cap e(u) \neq \emptyset\), satisfy conditions \eqref{cell_1_condition_1} and \eqref{cell_1_condition_2} and are non-critical.
\end{lemma}
\begin{proof}
Note that \(\be\) could occur as a face of \(2\)-cells of the following forms:

i) \ \(\al = e(v) \x e(u)\)

ii) \ \(\al = e \x e(u), \, e \in D_{v}\)

iii) \ \(\al = e \x e(u), \, e \in T_{v}\)

The value of \(\al\) might have been modified along with one of its faces in (i). In (ii), no value is modified. In (iii), the value of \(\al\) is modified along with a \(1\)-face of \(\al\) different from \(\be\).

We first show that \(\tilde{f}_2(\al) > \tilde{f}_2(\be)\) in cases (ii) and (iii), irrespective of the modifications made in Lemma \ref{lemma5}. Then, we consider the \(2\)-cells of type (i) and the faces of \(\be\).

(ii) \ \(\tilde{f}_{2}(\al) = f_{1}(e) + f_{1}(u) > f_{1}(v) + f_{1}(u) + 1 \geq \tilde{f}_2(\be)\quad \text{since } e\in D_v \text{, therefore } f_1(e)>f_1(v)+1.\)

(iii) \ \(\tilde{f}_2(\al) = f_{1}(e) + f_{1}(u)+1 > f_{1}(v) + f_{1}(u) + 1 \geq \tilde{f}_2(\be)\quad \text{since } e\in T_v.\)
Therefore, \(e=e(w)\) for some \(w\in \operatorname{V}(\g)\), where \(v=\tau(w)\) and
\[
f_2(e)=f_1(w)>f_1(\tau(w))=f_1(v).
\]

In cases (ii) and (iii), we obtain strict inequalities irrespective of the modifications made in Lemma \ref{lemma5}.

Now, we only need to check case (i) and the faces of \(\be\).

In case (i), since \(e(u) \cap e(v) \neq \emptyset\), either \(u=v\) or \(u\neq v\).

If \(u=v\), then \(e(u)=e(v)\), which gives \(\al = C_2(e(v))\). We also note that the value of \(\al\) is not modified in Lemma \ref{lemma6}, and hence
\[
f_2(\al)=f_2(\be).
\]

We only need to check the faces of \(\be\). The faces of \(\be\) are the following cells:
\[
\ka_1 = v \x v \text{ and } \ka_2 = v \x \tau(v).
\]

\[
f_{2}(\ka_1) = -4 < f_{2}(\be),
\]
because \(\ka_1\) is a \(0\)-cell in \(\Delta\).

Also,
\[
f_{2}(\ka_2) < 2f_1(v) = f_{2}(\be).
\]

In this case, the conditions \eqref{cell_1_condition_1} and \eqref{cell_1_condition_2} are satisfied, and \(\be\) is a non-critical cell.

If \(u \neq v\) and the value of \(\be\) was modified along with \(\al\) in Lemma \ref{lemma5}, we get the following relation:
\begin{equation}
\label{lemma6 eqn 1}
\tilde{f}_{2}(\be) = \tilde{f}_{2}(\al) = f_{1}(u) + f_1(v) + 1.
\end{equation}

For the faces of \(\be\), we get the following relation:
\[
\tilde{f}_{2}(\be) > {f}_{2}(v \x u) > {f}_{2}(v \x \tau(u)) \geq -4.
\]

If the value of \(\be\) was not modified in Lemma \ref{lemma5}, then we get
\begin{equation}
\label{lemma 6 eqn 2}
\tilde{f}_{2}(\be) = f_{1}(u) + f_1(v) < f_1(u) + f_1(v) + 1 = \tilde{f}_{2}(\al).
\end{equation}

For the faces of \(\be\), we have
\[
\tilde{f}_{2}(\be) = {f}_{2}(v \x u) > {f}_{2}(v \x \tau(u)) \geq -4.
\]

The purpose of including \(-4\) in the inequality is to emphasize that \(v \x \tau(u)\) could also possibly occur in \(\Delta\).

Therefore, the conditions \eqref{cell_1_condition_1} and \eqref{cell_1_condition_2} are satisfied, and the cell \(\be\) is non-critical.

\end{proof}
\begin{lemma}
    The \(1\)-cells of the form \(\be = v\x e\), where \(e \in \g - T\) and \(e(v)\cap e \neq \emptyset \) satisfy the conditions \eqref{cell_1_condition_1} and \eqref{cell_1_condition_2}. Furthermore, \(\be\) is critical when \(v=1\), otherwise it is non-critical.
\end{lemma}
\begin{proof}
Let \(e=(i,j)\), where \(i<j\).

Note that \(\be\) could occur as a face of \(2\)-cells of the following forms:

i) \ \(\al = e(v) \x e\)

ii) \ \(\al = e_i \x e, \, e_i \in D_{v}\)

iii) \ \(\al = e_i \x e, \, e_i \in T_{v}\)

Note that the values of \(f_2\) on \(\al\) and \(\be\) were not modified in cases (i), (ii), and (iii).

In case (i),
\[
{f}_{2}(\al) = f_{1}(v) + {f}_{1}(e) = {f}_{2}(\be).
\]

In cases (ii) and (iii),
\[
{f}_{2}(\al) = f_1(e_i)+f_1(e)> f_{1}(v) + {f}_{1}(e) = {f}_{2}(\be),
\]
because
\[
f_1(e_i)>f_1(v) \quad \text{for } e_i\in D_v \text{ or } e_i \in T_v.
\]

For the boundary of \(\be\), we have
\[
\tilde{f}_{2}(v \x i), \tilde{f}_{2}(v \x j)
< f_{1}(v) + {f}_{1}(e)
= \tilde{f}_{2}(\be).
\]

Therefore, the conditions \eqref{cell_1_condition_1} and \eqref{cell_1_condition_2} are satisfied.

If \(v \neq 1\), then \(e(v)\) exists, and hence the \(2\)-cell in case (i) is valid. Therefore, \(\be\) is non-critical.

If \(v=1\), then \(\be\) is critical.
\end{proof}

\begin{lemma}
    The \(1\)-cells of the form \(\be = v\x e(u)\), where \(e(u)\in T\) and \(e(u)\cap e(v) = \emptyset \), satisfy conditions \eqref{cell_1_condition_1} and \eqref{cell_1_condition_2}, and are non-critical.
\end{lemma}

\begin{proof}
The cell \(\be\) could occur as a face of \(2\)-cells of the following forms:

i) \(\al =e(v)\x e(u)\)

ii) \(\al =e\x e(u),\, e \in D_v\)

iii) \(\al =e\x e(u),\, e\in T_v\)

For the \(2\)-cells in (ii), the values of \(f_2\) on the cells \(\al\) were not modified. Similarly, in (iii), the value of \(f_2\) on \(\al\) was modified along with one of its faces different from \(\be\), because \(e \in T_v\). Hence, there exists some \(w\in \operatorname{V}(\g)\) such that \(v=\tau(w)\). The value of \(f_2\) was modified on \(\al\) and possibly on \(w\x e(u)\), but not on
\[
\tau(w)\x e(u)=v\x e(u)
\]
by Lemma \ref{lemma5}. Finally, in (i), the value of \(f_2\) was modified on \(\al\) along with either \(v \x e(u)\) or \(u\x e(v)\).

\textbf{If the values of \(\al\) and \(\be\) were modified in case (i):}

\begin{equation}
\label{lemma9eqn1}
\tilde{f}_2(\al)= \tilde{f}_2(\be)= f_1(v)+f_1(u)+1.
\end{equation}

In cases (ii) and (iii), we get
\begin{equation}
\label{lemma9eqn2}
f_2(\al) = f_1(e) + f_1(u) \geq f_1(v) + f_1(u) + 2
> f_1(v) + f_1(u) + 1 = \tilde{f}_2(\be),
\end{equation}
because \(f_1(e)\geq f_1(v)\) for \(e\in D_v\) or \(e\in T_v\).

Therefore, \(\be\) satisfies condition \eqref{cell_1_condition_1}.

Now, for the faces of \(\be\), we get the following inequalities:
\begin{equation}
\label{lemma9eqn3}
f_2(v\x u)= f_1(u)+f_1(v) < \tilde{f}_2(\be),
\end{equation}
and
\begin{equation}
\label{lemma9eqn4}
f_2(v \x \tau(u)) = f_1(v) + f_1(\tau(u))
< f_1(v) + f_1(u) < \tilde{f}_2(\be).
\end{equation}

Hence, condition \eqref{cell_1_condition_2} is satisfied.

Note that if the value of \(\be\) was not modified in Lemma \ref{lemma5}, then the equality in \eqref{lemma9eqn1} becomes a strict inequality, the inequality in \eqref{lemma9eqn2} remains a strict inequality, the inequality in \eqref{lemma9eqn3} becomes an equality, and \eqref{lemma9eqn4} remains a strict inequality. Therefore, in both cases, the conditions \eqref{cell_1_condition_1} and \eqref{cell_1_condition_2} are satisfied, and \(\be\) is non-critical in either situation.
\end{proof}
\begin{lemma}
    The \(1\)-cells of the form \(\be = v\x e\), where \(e \in \g - T\) and \(e(v)\cap e = \emptyset \) satisfy conditions \eqref{cell_1_condition_1} and \eqref{cell_1_condition_2} and are critical when \(v = 1\), otherwise are non-critical.
\end{lemma}

\begin{proof}
The cell \(\be\) could occur as a face of \(2\)-cells of the following forms:

(i) \(\al =e(v)\x e\)

(ii) \(\al =e_i\x e,\, e_i \in D_v\)

(iii) \(\al =e_i\x e,\, e_i\in T_v\)

In all cases (i), (ii), and (iii), the values of \(f_2\) on \(\al\) and \(\be\) were not modified in Lemma \ref{lemma5}. Therefore, we get
\[
{f_2}(\be)=f_1(v)+f_1(e).
\]

In case (i), we have
\[
f_2(\al)=f_1(v)+f_1(e)=f_2(\be).
\]

In cases (ii) and (iii), we have
\[
f_2(\al)=f_2(e_i\x e)=f_1(e_i)+f_1(e)>f_1(v)+f_1(e)=f_2(\be),
\]
because \(f_1(e_i)>f_1(v)\) whenever \(e_i\in D_v\) or \(e_i\in T_v\).

Therefore, condition \eqref{cell_1_condition_1} is satisfied.

Now, let \(e=(i,j)\). The faces of \(\be\) are of the form \(v\x i\) and \(v\x j\). Note that
\[
f_1(e)>f_1(i) \quad \text{and} \quad f_1(e)>f_1(j).
\]
For each \(0\)-cell, we get
\[
f_2(v\x i)=f_1(v)+f_1(i)<f_1(v)+f_1(e)=f_2(\be),
\]
and
\[
f_2(v\x j)=f_1(v)+f_1(j)<f_1(v)+f_1(e)=f_2(\be).
\]

Hence, condition \eqref{cell_1_condition_2} is also satisfied. Moreover, we note that
\[
\be \text{ is a critical cell } \iff v=1.
\]
\end{proof}
\begin{lemma}
    The \(0\)-cells of the form \(\ka = u\x v \), where \(u\neq v\) and \(u\neq 1, v \neq 1\) satisfy condition \eqref{cell_0_condition} and are non-critical.
\end{lemma}

\begin{proof}
The value of \(f_2\) on \(\ka\) is given by
\[
{f}_{2}(\ka)=f_{1}(u)+f_{1}(v).
\]
Furthermore, \(\ka\) could occur as the boundary of the following cells:

1) \ \(u \x e(v)\):
\begin{equation}
\label{lemma 10 eqn 1}
{f}_{2}(u \x e(v))={f}_{2}(\ka).
\end{equation}

2) \ \(u \x e,\, e \in {T}_{v}\) or \(e \in D_{v}\):
\[
{f}_{2}(u \x e)=f_{1}(u)+f_{1}(e)>f_{1}(u)+f_{1}(v)={f}_{2}(\ka).
\]

3) \ \(v \x e(u)\):
\begin{equation}
\label{lemma 10 eqn 2}
{f}_{2}(v \x e(u))=f_{1}(u)+f_{1}(v)={f}_{2}(\ka).
\end{equation}

4) \ \(v \x e,\; e \in {T}_{u}\) or \(e \in D_{u}\):
\[
{f}_{2}(v \x e)=f_{1}(v)+f_{1}(e)>f_{1}(v)+f_{1}(u)={f}_{2}(\ka).
\]

We obtain equalities in \eqref{lemma 10 eqn 1} and \eqref{lemma 10 eqn 2}. We note that in Lemma \ref{lemma5}, we considered \(2\)-cells of the form
\[
\al=e(u)\x e(v),
\]
where \(u\neq v\). assume \(u<v\), then the value of \(\al\) together with one of its faces \(u\x e(v)\) is modified.

Therefore, one of the equalities becomes a strict inequality. Hence, condition \eqref{cell_0_condition} is satisfied, and \(\ka\) is a non-critical cell.
\end{proof}

\begin{lemma}
    The \(0\)-cells of the form \(\ka = 1\x v \), where \(v\neq 1\) satisfy condition \eqref{cell_0_condition} and are non-critical.
\end{lemma}
\begin{proof}
The value of \(f_2\) on \(\ka\) is given by
\[
{f}_2(1 \x v) = {f}_1(1) + {f}_1(v).
\]
Furthermore, \(\ka\) could occur as a face of the following \(1\)-cells:

1) \(v\x e\), with \(e \in D_1\) or \(e \in T_1\):
\[
{f}_2(v \x e) = {f}_1(v) + {f}_1(e) > {f}_2(\ka),
\]
because \({f}_1(e) > {f}_1(1)\).

2) \(1\x e\), with \(e \in D_v\) or \(e \in T_v\):
\[
{f}_2(1 \x e) = {f}_1(1) + {f}_1(e) > {f}_2(\ka),
\]
because \({f}_1(e) > {f}_1(v)\).

3) \(1\x e(v)\):
\[
{f}_2(1 \x e(v)) = {f}_1(1) + {f}_1(v) = {f}_2(\ka).
\]

Note that there is no \(1\)-cell of the form \(e(1)\), because \(1\) is the only critical \(1\)-cell of \(\g\).

Therefore, \(\ka\) satisfies condition \eqref{cell_0_condition} and is critical.
\end{proof}
We conclude that \(\tilde{f}_2\) gives a discrete Morse function on \(\operatorname{Sym}^2(\g)\) with all cells in \(\Delta\) being critical. Equivalently, the restriction of \(\tilde{f}_2\) to \(C_2(\g)\) gives an open discrete Morse function \cite{KnudsonScoville2026}.

\section{Computation of Borel--Moore Homology for \texorpdfstring{$C_2(\g)$}{C2(g)}}

We summarize the information about the critical cells determined by the open discrete Morse function in the following remark.

\begin{remark}
\label{remark2}
\begin{itemize}
    \item A \(2\)-cell \(\al\) of \(C_2(\g)\) is critical \(\iff\) either \(\al=e\x \tilde{e}\) for some distinct \(1\)-cells \(e,\tilde{e} \in \g -T\) or \(\al=C_2(e)\) for some \(1\)-cell \(e\in\g-T\).
    \item A \(1\)-cell \(\be\) of \(C_2(\g)\) is critical \(\iff\) \(\be=1\x e\) for some \(1\)-cell \(e\in \g - T\).
    \item None of the \(0\)-cells is critical.
\end{itemize}
\end{remark}

Using this remark, we obtain the following Morse complex:
\[
\cdots \xrightarrow{0}
0
\xrightarrow{\partial_{3}}
\mathcal{M}_2
\xrightarrow{\partial_2}
\mathcal{M}_{1}
\xrightarrow{0}
0,
\]
where
\[
\mathcal{M}_2=\langle e\x\tilde{e},C_2(\bar e)\mid e,\tilde{e},\bar e\in\g-T \text{ with }e\neq\tilde{e}\rangle,\qquad
\mathcal{M}_1=\langle 1\x e\mid e\in\g-T\rangle,
\]
and \(\mathcal{M}_n=0\) for \(n=0\) or \(n\ge3\).

We also observe the following:
\[
|\mathcal{M}_2|=\frac{r(r+1)}{2}
\quad\text{and}\quad
|\mathcal{M}_1|=r,
\qquad\text{where } r=|\g-T|,\text{ i.e., the number of }1\text{-cells in }\g-T.
\]

\begin{remark}
\label{remark on betti numbers}
If \(\g\) is connected, then \(r\) is equal to the first Betti number of \(\g\), i.e., \(r=b_1\), where \(b_1\) denotes the first Betti number of \(\g\) (equivalently, the number of holes in \(\g)\).
\end{remark}
Also note that, in this case 
\[
b_1=|\operatorname{E}(\g)|-|\operatorname{V}(\g)|+1.
\]

We immediately note that if \(\g\) is a tree, then, by Remark~\ref{remark2}, the discrete Morse function defined above has no critical cells in any dimension. Therefore,
\[
H^{BM}_n(C_2(\g))\cong 0,\qquad \forall\, n\ge0.
\]

Based on Lemma~\ref{lemma4}, we make the following remark.

\begin{remark}
\label{remark1}
The cells \(\{v \x e,\; e(v) \x e\}\) form a regular pair for each \(v \neq 1\) and \(e\in \g-T\).
\end{remark}
Using Remark~\ref{remark2}, we note that any gradient upper path starting with a \(1\)-cell can only contain \(2\)-cells of the following forms:

\textbf{Type 1:} \(e(v) \x e(u)\), where \(e(v), e(u) \in T\),

\textbf{Type 2:} \(e(v) \x e\), where \(e(v) \in T\) and \(e \in \g - T\).

We also identify \(1\)-cells as follows:

\textbf{Type a:} \(v\x e(u)\), where \(e(u) \in T\),

\textbf{Type b:} \(v\x e\), where \(e \in \g - T\).

By Remark \ref{remark1}, a \(1\)-cell of \textbf{Type a} with \(v\neq1\) can only be followed by a 2-cell of \textbf{Type 1}, which is followed by a 1-cell of \textbf{Type a} again, and so on. We note that cells of \textbf{Type b} cannot occur in this path.

Therefore, a path containing 1-cells of \textbf{Type a} does not lead to a critical 1-cell, because critical 1-cells belong to \textbf{Type b}.

Now, given a \(2\)-critical cell, its boundary contains only cells of \textbf{Type b}, and there is only one unique path avoiding 1-cells of type a, which leads to a critical 1-cell.

Let \(v\x e\) be a cell of \textbf{Type b}, with \(v\neq 1\)(when \(v=1\), we get the constant path); the path starting with \(v \x e\) and ending with a critical \(1\)-cell is unique and is as follows:
\[
\gamma = v \x e \nearrow e(v) \x e \searrow \tau(v) \x e \nearrow e(\tau(v)) \x e\searrow \cdots
\nearrow e(2)\x e\searrow 1 \x e. 
\]
Note that the sequence \(v,e(v),\tau(v),e(\tau(v)),\cdots\) eventually leads to \(1\x e\). 
It is easy to see that the multiplicity of this path \(\mu(\gamma) = 1\), because for each elementary upper path of the form \(w\x e \nearrow e(w)\x e \searrow \tau \left(w\right)\x e\), we get:
\[
\mu\left(w\x e \nearrow e(w)\x e \searrow \tau \left(w\right)\x e\right)=-i_{w\x e , e(w)\x e }\cdot i_{\tau(w)\x e,e(w)\x e}=1,
\]
 where \(i\) is the incidence number. Using multiplicity of \(\mu\), we conclude \(\mu(\gamma)=1\) for all such paths.

\textbf{1-cells in boundary of the 2-critical cells:}

Let \(\al \in C_2(\g)\) be a \(2\)-critical cell. If \(\al =e\x\tilde{e}\), with \(e, \tilde{e} \in \g - T\) such that \(e \neq \tilde{e}\).
Then

\begin{equation}
\label{eq:Morse_boundary_1}
 \partial_2\, \al = -1 \x \tilde{e} - 1 \x e + 1 \x \tilde{e} + 1 \x e = 0,
\end{equation}
considering the multiplicity and orientation.

If \(\al = C_2(e)\) for some \(e\in C_2(\g)\), then
\begin{equation}
\label{eq:Morse_boundary_2}
\partial_2 \,\al = 2(1 \x e) \text{ or }\partial_2\, \al = - 2(1 \x e),
\end{equation}
because, both the \(1\)-faces \(i\x e\) and \(j\x e\) of \(\al\) in the open complex have same incidence numbers.

\begin{theorem}
Let \(\g\) be a simple and connected graph. Then the Borel--Moore homology of the corresponding unordered \(2\)-particle configuration space \(C_2(\g)\) is given by
\[
H^{BM}_{n}(C_2(\g)) \cong
\begin{cases}
\mathbb{Z}^{\frac{b_1(b_1-1)}{2}}, & n=2,\\[4pt]
\mathbb{Z}_2^{\,b_1}, & n=1,\\[4pt]
0, & \text{otherwise}.
\end{cases}
\]
\end{theorem}
\begin{proof}
Using \eqref{eq:Morse_boundary_1} and \eqref{eq:Morse_boundary_2}, together with Remark \ref{remark on betti numbers}, we obtain
\begin{align*}
H^{BM}_2(C_2(\g)) 
&\cong \operatorname{Ker}\partial_2,\\
&= \langle  e\x\tilde{e}\,|\,e,\tilde{e}\in \g -T  \text{ and } e \neq \tilde{e}\rangle,\\ 
&\cong\mathbb{Z}^{\frac{b_1(b_1+1)}{2}-b_1},\\ 
&\cong \mathbb{Z}^{\frac{b_1(b_1-1)}{2}}.
\end{align*}
and 
\begin{align*}
H^{BM}_1(C_2(\g)) 
&\cong \frac{\operatorname{Ker}\partial_1}{\operatorname{Im\partial_2}}\, , \\
&=\frac{\langle  1\x{e}\,|e\in \g -T  \rangle}{\langle \pm 2( 1\x{e})\,|e\in \g  - T  \rangle}\, ,\\ 
&\cong \mathbb{Z}_2^{b_1}\, .
\end{align*}

Therefore, we conclude that 
\[
H^{BM}_{n}(C_2(\g)) \cong
\begin{cases}
\mathbb{Z}^{\frac{b_1(b_1-1)}{2}}, & n=2,\\[4pt]
\mathbb{Z}_2^{\,b_1}, & n=1,\\[4pt]
0, & \text{otherwise}.
\end{cases}
\]
\end{proof} 
As an example, we consider the graph given in Figure \ref{fig:G_k_graph}, which is the skeleton of a genus \(k\) orientable surface. Note that it is a simple and connected graph.
\begin{figure}[ht]
\centering

\begin{tikzpicture}[scale=1]

\foreach \i in {1,...,12} {
    \node[circle,fill,inner sep=1.5pt] (v\i) at (\i,0) {};
}

\node[below, font=\tiny] at (1,-0.2) {1};
\node[below, font=\tiny] at (2,-0.2) {2};
\node[below, font=\tiny] at (3,-0.2) {3};
\node[below, font=\tiny] at (4,-0.2) {4};
\node[below, font=\tiny] at (5,-0.2) {5};
\node[below, font=\tiny] at (6,-0.2) {6};
\node[below, font=\tiny] at (7,-0.2) {7};
\node[below, font=\tiny] at (8,-0.2) {8};

\node[below, font=\tiny] at (9,-0.2) {\(4k-3\)};
\node[below, font=\tiny] at (10,-0.2) {\(4k-2\)};
\node[below, font=\tiny] at (11,-0.2) {\(4k-1\)};
\node[below, font=\tiny] at (12,-0.2) {\(4k\)};

\draw (1,0) -- (12,0);

\begin{knot}[clip width=5, consider self intersections=true]

\strand (v1) to[out=60,in=120,looseness=1.2] (v3);
\strand (v2) to[out=60,in=120,looseness=1.2] (v4);

\flipcrossings{1-2}

\end{knot}

\begin{knot}[clip width=5, consider self intersections=true]

\strand (v5) to[out=60,in=120,looseness=1.2] (v7);
\strand (v6) to[out=60,in=120,looseness=1.2] (v8);

\flipcrossings{1-2}

\end{knot}

\begin{knot}[clip width=5, consider self intersections=true]

\strand (v9)  to[out=60,in=120,looseness=1.2] (v11);
\strand (v10) to[out=60,in=120,looseness=1.2] (v12);

\flipcrossings{1-2}

\end{knot}

\draw[white, line width=6pt] (8.3,0) -- (8.7,0);
\node at (8.5,0) {\(\cdots\)};

\node at (2.5,1.2) {\small 1};
\node at (6.5,1.2) {\small 2};
\node at (10.5,1.2) {\small k};
\end{tikzpicture}
\caption{\(G_k\) graph}
\label{fig:G_k_graph}

\end{figure}
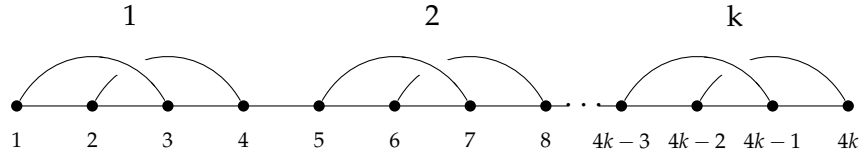

The first Betti number for the graph \(G_k\) is \(2k\), therefore, we get:

\[
H^{BM}_{n}(C_2(G_k)) \cong
\begin{cases}
\mathbb{Z}^{{k(2k-1)}}, & n=2,\\[4pt]
\mathbb{Z}_2^{\,2k}, & n=1,\\[4pt]
0, & \text{otherwise}.
\end{cases}
\]

\section{Future Directions}
This work forms the basis of several additional lines of inquiry. For instance, the authors of \cite{BlanchetPalmerShaukat2025} suggest that in many cases, configurations on surfaces may be reduced to configurations on graphs. In such cases, our main result can be applied to compute the Borel--Moore homology of the corresponding graphs. Analogously, our work can be naturally extended to the study of $k$-particle ordered/unordered configuration spaces.

\section*{Acknowledgments}

The author thanks Prof. Christian Blanchet and Prof. Haniya Azam for supervising the Master's thesis, which set out to learn tools for Topological Data Analysis, with Discrete Morse Theory being one of them and from which this work originated. The author thanks HEC for the grant (NRPU Ref No. 20-16557) which provided necessary funding during the Master's thesis. Additionally, the author is grateful to Prof. Blanchet for his continued guidance beyond the completion of the thesis, valuable discussions, and for suggesting a problem that forms the basis of this extension.

\bibliographystyle{abbrv}
\bibliography{reference}
\end{document}